\documentclass[a4paper, 11pt]{amsart}

\usepackage{amsmath,amssymb,amsthm,bm}
\usepackage{tikz-cd}
\usepackage{dutchcal} 

\usepackage{graphicx}
\usepackage[T1]{fontenc}
\usepackage{newpxtext}  
\usepackage{newpxmath}  
\usepackage{multirow,multicol}

\usepackage{microtype}			
\usepackage{float}

\usepackage[colorlinks=true,linkcolor=blue,urlcolor=blue, citecolor=blue]{hyperref}

\usepackage[inner=2.3cm,outer=2.3cm, bottom=3.3cm]{geometry}

\newcommand{\mc}[1]{\mathcal{#1}}
\newcommand{\mr}[1]{\mathrm{#1}}
\newcommand{\mf}[1]{\mathfrak{#1}}
\newcommand{\mb}[1]{\mathbb{#1}}

\newcommand{\ideala}[1]{\langle{#1}\rangle}
\newcommand{\kh}[1]{\left(#1\right)}

\newcommand{\ssm}[1]{\sum_{#1}}

\newcommand{\td}[1]{\bm{#1}}
\newcommand{\tld}[1]{\tilde{#1}}
\newcommand{\tx}[1]{\text{#1}}

\newcommand{\fq}{\mathbb{F}_q}

\theoremstyle{definition}
\newtheorem{theorem}{Theorem}[section]
\newtheorem{lemma}[theorem]{Lemma}
\newtheorem{corollary}[theorem]{Corollary}
\newtheorem{proposition}[theorem]{Proposition}
\newtheorem{claim}{Claim}
\theoremstyle{definition}

\newtheorem{model}{Modeling}
\newtheorem*{problem}{Problem}

\usepackage{comment}
\usepackage[colorinlistoftodos]{todonotes}
\newtoggle{comments}
\toggletrue{comments}

\iftoggle{comments}{
	\newcommand{\pg}[2][]{\todo[color=green!30,#1]{\textsf{PG:} #2}}
}{
	\newcommand{\pg}[2]{}
	
}
\title{On the Relations Among Algebraic Models for the MinRank Problem}

\author{Peigen Li}
\address{Beijing Institute of Mathematical Sciences and Applications, China}
\email{lpg22@bimsa.cn}

\author{Siyong Tao}
\address{Tsinghua University, China}
\email{taosy22@mails.tsinghua.edu.cn}

\begin{document}
	
	\footnote{This work is supported by NSFC 126011023.}
	\begin{abstract}
		In this article, we study the relations between algebraic models for the MinRank problem, which is a well-known NP-hard problem. We establish connections between the minors modeling, the Kipnis-Shamir modeling, and the support minors modeling, providing insights into their equivalences and differences. Our results contribute to a deeper understanding of the algebraic structures underlying the MinRank problem and may inform future research on efficient algorithms for solving it.
	\end{abstract}
	\maketitle
	\section{Introduction}
	Throughout this paper, let $\fq$ denote the finite field of order $q$ ($q$ is a power of a prime number $p>0$). The MinRank problem serves as a security cornerstone for many post-quantum public-key cryptosystems, including, but not limited to, code-based \cite{gaborit2015complexity,loidreau2017new} and multivariate-based cryptography \cite{patarin1996hidden,kipnis1999unbalanced,ding2005rainbow}. The MinRank problem is formulated as follows. 
	\begin{problem}
		Given $k$ matrices $\{M_i\}_{i=1}^k$ of size $m\times n$ with all entries in $\fq$ and a positive integer $r<$ $\min\{m,n\}$, find a nontrivial vector $\td{a}=(a_1,\dots,$ $a_k)\in\fq^k$ such that
		\begin{equation*}
			\tx{rank}\,\kh{\td{M_a}:=\sum^k_{i=1}a_iM_i}\leq r.
		\end{equation*}
	\end{problem}
	Although Buss et al. proved that the MinRank problem is NP-hard in general \cite{buss1999computational}, various modeling approaches for solving it have continued to emerge.
	
	From the algebraic point of view, MinRank admits several standard polynomial modelings. A first approach uses the vanishing of all $(r+1)\times(r+1)$ minors of $\td{M_x}$. A second approach is the Kipnis--Shamir (KS) modeling, where low rank is encoded by the existence of a sufficiently large kernel and one introduces auxiliary variables for a basis of that kernel. A third approach, often called the support-minors modeling, introduces auxiliary variables describing a rank factorization and derives bilinear equations from maximal minors of an auxiliary support matrix.
	
	Understanding these relations is useful for both conceptual and practical reasons. Conceptually, it clarifies which algebraic attacks are genuinely different and which are simply different presentations of the same geometric condition. Practically, it helps explain when one formulation may yield a smaller system, a better elimination behavior, or a more convenient local chart for computation.
	
	In \cite{guo2022algebraic,GSM}, the authors found that the equations in Kipnis-Shamir modeling can be derived from the support minors modeling, it's actually the statement in Theorem \ref{thm:main thm} (1) in the following,
	and the equations in the minors modeling can also be derived from the Kipnis-Shamir modeling.
	
	Before stating the algebraic modelings precisely, we need one bit of terminology. Let $R=\fq[x_1,$ $\dots,x_k]$ be the ring of polynomials in $k$ variables $x_1,\dots,x_k$ with coefficients in $\fq$. We denote by $\fq^k$ the space of $k$-tuples of elements of $\fq$. Given any ideal $I\subseteq R$, we let
	\begin{equation*}
		\mb{V}(I)=\{\td{a}=(a_1,\dots,a_k)\in\fq^k:f(\td{a})=0\;\tx{for all}\;f\in I\}\subseteq\fq^k.
	\end{equation*}
	We call $\mb{V}(I)$ the affine algebraic set defined by $I$. Given a subset $V$ of $\fq^n$, the ideal
	\begin{equation*}
		\mb{I}(V)=\ideala{f\in R:f(\td{a})=0\;\tx{for all}\;\td{a}\in V}\subseteq R.
	\end{equation*}
	is called the vanishing ideal of $V$. The affine $\fq$-Nullstellensatz \cite[Theorem 2.3]{Null} states that for an affine algebraic set $\mb{V}(I)$ defined by some ideal $I\subseteq R$, we have
	\begin{equation*}
		\mb{I}(V)=I+\ideala{x_i^q-x_i:1\leq i\leq n}.
	\end{equation*}
	This means that for any ideal $I$ of $R$, $I+\ideala{x_i^q-x_i:1\leq i\leq n}=\mb{I}(\mb{V}(I))$ is a radical ideal. 
	
	Consider a MinRank instance $(n,m,r,M_1,\dots,M_k)$. $\td{a}=(a_1,\dots,a_k)$ is a solution of the MinRank problem if and only if all $(r+1)$-minors of the polynomial matrix $\td{M_x}$ simultaneously vanish on this point. Thus the Minors modeling which has been presented in \cite{Minors} is obtained by considering the algebraic system of all $(r+1)$ minors of $\td{M_x}$.
	\begin{model}[Minors Modeling]
		Let $I_{\tx{minors}}=I_{r+1}(\td{M_x})$ be the ideal of $R$ generated by all $(r+1)$-minors of the polynomial matrix $\td{M_x}$. The locus of the rank defect is the algebraic set $\mb{V}(I_{r+1}(\td{M_x}))$ in $\fq^k$. We denote it by $V_{\tx{minors}}$.
		\begin{align*}
			V_{\tx{minors}}=\{\td{a}\in\fq^k:\Delta(\td{a})=0\;\tx{for all}\;(r+1)\tx{-minors $\Delta$ of}\;\td{M_x}\}.
		\end{align*}
	\end{model}
	$\td{a}=(a_1,\dots,a_k)\in\fq^k$ is a solution of the MinRank problem if and only if there are at least $n-r$ linearly independent vectors in the right kernel of $\td{M_a}$. Equivalently, there is an $n\times (n-r)$ matrix $P$ over $\fq$ of rank $n-r$ such that $\td{M_a}P=\td{0}_{m\times(n-r)}$ (the zero matrix). If the first $n-r$ rows of $P$ form an invertible matrix $Q$, then $MPQ^{-1}=\td{0}_{m\times(n-r)}$.
	We can specialize $P$ with its first $n-r$ rows equal to the identity matrix $I_{n-r}$.
	
	\begin{model}[Kipnis-Shamir Modeling,\cite{KS}]
		We search for solutions $\td{a}\in\fq^k$ and $\td{b}=(b_{1,1},\dots,$ $b_{r,n-r})\in\fq^{r(n-r)}$ such that
		\begin{equation}\label{eq:KS sys}
			\td{M_a}\cdot\begin{pmatrix}
				I_{n-r}\\
				\tx{Var}(\td{b})
			\end{pmatrix}=\td{0}_{m\times (n-r)},\quad\tx{Var}(\td{b})=\begin{pmatrix}
				b_{1,1}&\cdots&b_{1,n-r}\\
				\vdots&\ddots&\vdots\\
				b_{r,1}&\cdots&b_{r,n-r}
			\end{pmatrix}.
		\end{equation}
		Let $R_{\tx{KS}}$ be the polynomial ring $\fq[x_1,\dots,x_k,y_{1,1},\dots,y_{r,n-r}]$, and $I_{\tx{KS}}$ be the ideal generated by
		\begin{equation}\label{eq:KS poly}
			f_{i,j}(\td{x},\td{y})=u_{i,j}(\td{x})+\sum^r_{l=1}u_{i,n-r+l}(\td{x})y_{l,j},\quad i\in[m],j\in[n-r]
		\end{equation}
		Here $u_{i,j}(\td{x})$ is the $(i,j)$-th entry of $\td{M_x}$.
		$f_{i,j}(\td{x},\td{y})$ is the $(i,j)$-th entry of the $m\times (n-r)$ matrix
		\begin{equation*}
			\td{M_x}\cdot\begin{pmatrix}
				I_{n-r}\\
				\tx{Var}(\td{y})
			\end{pmatrix},\quad\tx{Var}(\td{y})=\begin{pmatrix}
				y_{1,1}&\cdots&y_{1,n-r}\\
				\vdots&\ddots&\vdots\\
				y_{r,1}&\cdots&y_{r,n-r}
			\end{pmatrix}.
		\end{equation*}
		We denote by $V_{\tx{KS}}$ the set of all solutions $(\td{a},\td{b})\in\fq^k\times\fq^{r(n-r)}$ of the system \eqref{eq:KS sys}. We see that $V_{\tx{KS}}=\mb{V}(I_{\tx{KS}})$ (here $\fq^k\times\fq^{r(n-r)}$ is identified with $\fq^{k+n(r-n)}$).
	\end{model}
	
	If $(\td{a},\td{b})\in V_{\tx{KS}}$, then the right kernel of $\td{M_a}$ contains at least $n-r$ linearly independent vectors, so $\tx{rank}\,\td{M_a}\leq r$. However, there are some cases in which $\tx{rank}\,\td{M_a}\leq r$ for some $\td{a}\in\fq^r$ but we can't find any $\td{b}\in\fq^{r(n-r)}$ such that $(\td{a},\td{b})\in V_{\tx{KS}}$. If we write $\td{M_a}=(\td{M_{a,1}},\td{M_{a,2}})$ with $\td{M_{a,1}}$ is an $m\times (n-r)$ matrix and $\td{M_{a,2}}$ is an $m\times r$ matrix, the existence of such $\td{b}$ implies that
	\begin{equation*}
		\td{M_{a,1}}=-\td{M_{a,2}}\cdot \tx{Var}(\td{b}),\quad \td{M_a}=\td{M_{a,2}}\cdot B,\quad B=(-\tx{Var}(\td{b}),I_r).
	\end{equation*}
	The last $r$ columns of $\td{M_a}$ span the column space of $\td{M_a}$. And every row of $\td{M_a}$ can be expressed as a linear combination of rows of $B$. Let $(\td{M_a})_{i,*}=(u_{i,1}(\td{a}),\dots,u_{i,n}(\td{a}))$ be the $i$-th row of $\td{M_a}$. Then the $(r+1)\times n$ matrix below has rank $r$.
	\begin{equation}\label{eq:supp sys}
		B_i=\begin{pmatrix}
			u_{\td{a},i}\\
			B
		\end{pmatrix}.
	\end{equation}
	Therefore, all maximal $(r+1)$-minors of $B_i$ are equal to 0. Expanding along the first row of every maximal minor yields terms formed by $u_{i,j}(\td{a})$ multiplied by a maximal $r$-minor of $B$. For integers $r,n\geq 0$, we use the notation $\mc{P}_r(n)$ to denote the collection of all subsets of $[n]=\{1,\dots,n\}$ of size $r$. For $T\in\mc{P}_r(n)$, let $c_T$ be the maximal minor of $B$ corresponding to the columns that belong to $T$. For $S\in\mc{P}_{r+1}(n)$, the maximal minor of $B_i$ corresponding to the columns that belong to $S$ is equal to zero, then
	\begin{equation*}
		\ssm{j\in S}(-1)^{\tx{sgn}(j,S)}u_{i,j}(\td{a})c_{S\setminus\{j\}}=0.
	\end{equation*}
	Here $\tx{sgn}(j,S)=\#\{i\in S:i<j\}$ ($\#$ denotes the cardinality of the set). Moreover, for $i\in[r]$ and $j$ $\in[n-r]$, let $T_{i,j}$ be the subset $\{j,n-r+1,\dots,n\}\setminus\{n-r+i\}$.
	\begin{equation*}
		c_T=\begin{cases}
			1&T=T_0=\{n-r+1,\dots,n\}\\
			(-1)^ib_{i,j}&T=T_{i,j}.
		\end{cases}.
	\end{equation*}
	Take $S=S_j=\{j,n-r+1,\dots,n\}$ in \eqref{eq:supp sys}. $u_{i,j}(\td{a})+\sum^{n-r}_{l=1}(-1)^lu_{i,n-r+l}(\td{a})c_{T_{l,j}}=0$ directly implies that $f_{i,j}(\td{a},\td{b})=0$. These considerations lead to the following algebraic modeling, which is proposed in \cite{GSM}:
	
	\begin{model}[Support Minors Modeling]\label{mod:SM}
		Define the polynomial ring $R_{\tx{supp}}^{\tx{Gb}}=\fq[x_1,\dots,x_k,C_T:T\in\mc{P}_r(n)]$ by introducing new variable $C_T$ for each subset $T\in\mc{P}_r(n)$. We consider the ideal $I_{\tx{supp}}^{\tx{Gb}}$ generated by the bilinear polynomials 
		\begin{equation}\label{eq:SM poly}
			h_{i,S}(\td{x},\td{C})=\ssm{j\in S}(-1)^{\tx{sgn}(j,S)}u_{i,j}(\td{x})C_{S\setminus\{j\}},\quad i\in[m],S\in\mc{P}_{r+1}(n).
		\end{equation}
		Find $\td{a}\in\fq^k$ and $\td{c}=(c_T)_{T\in\mc{P}_r(n)}\in\fq^N$ ($N=\binom{n}{r}$) such that
		\begin{equation*}
			(\td{a},\td{c})\in V^{\tx{Gb}}_{\tx{supp}}=\mb{V}(I^{\tx{Gb}}_{\tx{supp}})\setminus(\fq^k\times\{\td{0}_N\}),\quad\td{0}_N=(0,\dots,0)\in\fq^N.
		\end{equation*}
	\end{model}
	
	The superscript $^{\tx{Gb}}$ means that we do not impose a condition that $c_{T_0}=1$. We will show that any solution $(\td{a},\td{c})\in V^{\tx{Gb}}_{\tx{supp}}$ can give the solution $\td{a}\in V_{\tx{minors}}$ via projection on the first $k$ coordinates, see Proposition \ref{pro:Vsup and Vmin}. We also consider the algebraic set 
	\begin{equation*}
		V_{\tx{supp}}=V^{\tx{Gb}}_{\tx{supp}}\cap \mb{V}(\ideala{C_{T_0}-1})=\mb{V}(I^{\tx{Gb}}_{\tx{supp}}+\ideala{C_{T_0}-1}).
	\end{equation*}
	For $(\td{a},\td{c})\in V_{\tx{supp}}$, we set $b_{i,j}=(-1)^ic_{T_{i,j}}$ for $i\in[r]$ and $j\in[n-r]$. Then $(\td{a},\td{b})\in V_{\tx{KS}}$, here $\td{b}=(b_{1,1},$ $,\dots,b_{r,n-r})$. Thus we get a correspondence between $V_{\tx{KS}}$ and $V_{\tx{supp}}$. Let $\sigma:\fq^k\times \fq^{N}\to\fq^k\times\fq^{N-1}$ be the projection map sending $(\td{a},\td{c})$ to $(\td{a},(c_T)_{T\in\mc{P}_r(n)\setminus \{T_0\}})$. $V_{\tx{supp}}$ is isomorphic to $\sigma(V_{\tx{supp}})=\mb{V}(I_{\tx{supp}})$, where $I_{\tx{supp}}\subseteq R_{\tx{supp}}=R[C_T:T\in\mc{P}_r(n)\setminus\{T_0\}]$ is obtained by setting $C_{T_0}=1$ in all elements of $I_{\tx{supp}}^{\tx{Gb}}$. 
	
	The relations between $V_{\tx{minors}},V_{\tx{KS}},V^{\tx{Gb}}_{\tx{supp}}$ and $V_{\tx{supp}}$ are clearly. However, they are all discrete and reducible algebra sets over the finite field $\fq$, so we cannot directly use the Hilbert Nullstellensatz to deduce the relationship between their vanishing ideals and defining ideals $I_{\tx{minors}},I_{\tx{KS}},I_{\tx{Supp}}^{\tx{Gb}}$ and $I_{\tx{supp}}$.
	
	\begin{theorem}\label{thm:main thm}\,
		\begin{enumerate}
			\item  $I_{\tx{minors}}\subseteq I_{\tx{KS}}\cap R$. Moreover, assume that $\Delta\neq 0$ is an $r$-minor of $\td{M_x}$ with columns indexed by $\{n-r+1,\dots,n\}$. Then $(I_{\tx{minors}})_\Delta=(I_{\tx{KS}})_\Delta\cap R_\Delta$.
			\item Let $\Phi:\tld{R}_{\tx{supp}}=R[C_{T_{i,j}}:i\in[m],j\in[n-r]]\to R_{\tx{KS}}$ be the $\fq$-algebra isomorphism which sends $C_{T_{i,j}}$ to $(-1)^iy_{i,j}$. We have $\Phi(I_{\tx{supp}}\cap\tld{R}_{\tx{supp}})=I_{\tx{KS}}$.
			\item For any $T\in\mc{P}_r(n)$, $C_T^{r+1}I_{\tx{minors}}\subseteq I^{\tx{Gb}}_{\tx{supp}}$, and
			\begin{align*}
				&I_{\tx{minors}}+\ideala{x_i^q-x_i:i\in[k]}\\=&\kh{(I^{\tx{Gb}}_{\tx{supp}}+\ideala{x_i^q-x_i,C_T^q-C_T:i\in[k],T\in\mc{P}_r(n)}):\ideala{C_T:T\in\mc{P}_r(n)}}\cap R.
			\end{align*}
		\end{enumerate}
	\end{theorem}
	The first theorem states the relations among the three ideals defined by different algebraic models. The second theorem states the relations among the algebraic sets defined by different algebraic models.
	
	The paper is organized as follows. In Sect. \ref{sec:preliminaries}, we introduce some notations and preliminaries. In Sect. \ref{sec:relations}, we study the relations among the ideals and algebraic sets defined by different algebraic models.

	\section{Notations and Preliminaries}\label{sec:preliminaries}
	We denote the ring of $m\times n$ matrices over $R$ by $\tx{Mat}_{m\times n}(R)$. A $1\times n$ (resp. $m\times 1$) matrix is called a row vector (resp. a column vector). We view $R^m$ as the module of column vectors of size $m$. The zero matrix of size $m\times n$ is denoted by $\td{0}_{m\times n}$ and the identity matrix of size $n\times n$ is denoted by $I_n$. We can also define the ring of $m\times n$ matrices over a commutative ring with a unit element $1$ in a similar way. Let $A=(a_{i,j})_{1\leq i\leq m,1\leq j\leq n}\in\tx{Mat}_{m\times n}(R)$. 
	\begin{itemize}
		\item For an ideal $I$ of $R$, $A\bmod I$ is the matrix $(a_{i,j}\bmod I)_{1\leq i\leq m,1\leq j\leq n}\in\tx{Mat}_{m\times n}(R/I)$.
		\item When $m=n$, $|A|$ is the determinant of the square matrix.
		\item Given any subsets $T_1\subseteq[m]$ and $T_2\subseteq[n]$, $A_{T_1,T_2}$ is the submatrix of $A$ formed from rows in $T_1$ and columns in $T_2$, i.e., $A_{T_1,T_2}=(A_{i,j})_{i\in T_1,j\in T_2}$. When $T_1=[m]$ (resp. $T_2=[n]$), we shall use the notation $A_{*,T_2}$ (resp. $A_{T_1,*}$).
		\item Fix an integer $r>0$. $I_r(A)$ is the ideal of $R$ generated by all $r$-minors of $A$, i.e.,
		\begin{equation*}
			I_r(A)=\ideala{|A_{T_1,T_2}|:T_1\in\mc{P}_r(m),T_2\in\mc{P}_r(n)}\subseteq R.
		\end{equation*}
		Especially, $I_1(A)$ is the ideal generated by all the entries of $A$.
	\end{itemize}
	
	\begin{claim}\label{clm:minor}
		Given $A=(a_{i,j})_{1\leq i,j\leq n}\in\tx{Mat}_{n\times n}(R)$ and $\td{b}=(b_1,\cdots,b_n)^T$, let $\tx{adj}(A)=(c_{i,j})_{1\leq i,j\leq n}$ be the adjugate of $A$. $(-1)^{i+j}c_{i,j}$ is the $(n-1)$-minor obtained from $A$ by leaving out the $j$-th row and the $i$-th column (we also call it the $(n-1)$-minor of $a_{i,j}$ in $A$). For $1\leq i\leq n$, we have
		\begin{equation*}
			\sum^n_{j=1}c_{i,j}b_j=(-1)^{i-1}\begin{vmatrix}
				b_1&a_{1,1}&\cdots&a_{1,i-1}&a_{1,i+1}&\cdots&a_{1,n}\\
				b_2&a_{2,1}&\cdots&a_{2,i-1}&a_{2,i+1}&\cdots&a_{2,n}\\
				\vdots&\vdots&\ddots&\vdots&\vdots&\ddots&\vdots\\
				b_n&a_{n,1}&\cdots&a_{n,i-1}&a_{n,i+1}&\cdots&a_{n,n}.
			\end{vmatrix}
		\end{equation*}
	\end{claim}

	\begin{claim}\label{clm:det}
		Fix a matrix $A=(a_{i,j})_{1\leq i,j\leq n}\in\tx{Mat}_{n\times n}(R)$. $\td{a}_1,\dots,\td{a}_n$ are columns of $A$. Suppose that there exist column vectors $\{\td{b}_i\}_{i=1}^r\subseteq R^n$ with $r<n$ generating $\{\td{a}_i\}_{i=1}^n$, i.e.,
		\begin{equation*}
			\td{a}_i=\sum^r_{j=1}u_{i,j}\td{b}_j,\quad u_{i,j}\in R.
		\end{equation*}
		Then $|A|=0$. Actually, $A$ can written in the form
		\begin{equation*}
			A=\begin{pmatrix}
				\td{b}_1&\cdots&\td{b}_r&\td{0}_n&\cdots\td{0}_n
			\end{pmatrix}\begin{pmatrix}
				u_{1,1}&u_{2,1}&\cdots&u_{n,1}\\
				\vdots&\vdots&\ddots&\vdots\\
				u_{1,r}&u_{2,r}&\cdots&u_{n,r}\\
				0&0&\cdots&0\\
				\vdots&\vdots&\ddots&\vdots\\
				0&0&\cdots&0
			\end{pmatrix}.
		\end{equation*}
		The two matrices on the right are of size $n\times n$. By the product property of determinants we see that $|A|=0$. Moreover, given column vectors $\{\td{c}_{i_l}\}_{l=1}^m\subseteq R^n$ with $1\leq i_1<\cdots<i_m\leq n$, let $A_1$ be the matrix with columns $\td{a}_1,\dots,\td{a}_{i_1}+\td{c}_{i_1},\dots,\td{a}_{i_m}+\td{c}_{i_m},\dots,\td{a}_n$, we have $|A_1|\in I=\sum^m_{l=1}I_1(\td{c}_{i_l})$ (by our notation, $I_1(\td{c}_{i_l})$ is the ideal of $R$ generated by all entries of the vector $\td{c}_{i_l}$), this is because
		\begin{equation*}
			|A_1|\bmod I=|A_1\bmod I|=|A\bmod I|=|A|\bmod I=0.
		\end{equation*}
	\end{claim}
	
	\begin{claim}\label{clm:I cap R}
		Consider an ideal $I=\ideala{f_1,\dots,f_m}$ in the polynomial ring $R[y_1,\dots,y_n]$ in indeterminates $y_i$. If there exist elements $u_1,\dots,u_n\in R$ such that $y_i-u_i\in I$, then $I\cap R=\ideala{f_i(u_1,\dots,u_n):1\leq i\leq m}$. Actually, we expand the polynomial $f_i$ after making the substitution $y_i=y_i-u_i+u_i$, then we get $f_i=f_i(u_1,\dots,u_n)+g_i(y_1-u_1,\dots,y_n-u_n)$, where $g$ is a polynomial without any constant term. So $\ideala{f_i(u_1,\dots,u_m):1\leq i\leq m}\subseteq I\cap R$. For the reverse inclusion, if $u=\sum^m_{i=1}g_if_i\in I\cap R$, evaluating the right-hand side on $(u_1,\dots,u_n)$ yields $u=\sum^m_{i=1}g_i(u_1,\dots,u_n)f_i(u_1,\dots,u_n)$.
	\end{claim}
	
	\section{Relations}\label{sec:relations}
	Consider a MinRank instance $\{M_i\}_{i=1}^k\in\tx{Mat}_{m\times n}(\fq)$ with target rank $r$, let $R=\fq[x_1,\dots,x_n]$ and $\td{M_x}=\sum^k_{i=1}x_iM_i$. $u_{i,j}(\td{x})\in R$ is the $(i,j)$-th entry of $\td{M_x}$. The sets of solutions of three modelings ($V_{\tx{KS}},V_{\tx{minors}},V_{\tx{supp}}^{\tx{Gb}}$ and $V_{\tx{supp}}$) are related to three algebraic sets defined by different ideals in polynomial rings. 
	
	\begin{enumerate}
		\item $V_{\tx{minors}}=\mb{V}(I_{\tx{minors}})\subseteq\fq^k$, where $I_{\tx{minors}}=I_{r+1}(\td{M_x})\subseteq R$.
		\item $V_{\tx{KS}}=\mb{V}(I_{\tx{KS}})\subseteq \fq^k\times\fq^{r(n-r)}$, where $I_{\tx{KS}}=\ideala{f_{i,j}:i\in [m],j\in [n-r]}\subseteq R_{\tx{KS}}=R[y_{1,1},\dots,y_{r,n-r}]$ and $f_{i,j}$ is as in \eqref{eq:KS poly}.
		\item $V^{\tx{Gb}}_{\tx{supp}}=\mb{V}(I^{\tx{Gb}}_{\tx{supp}})\setminus(\fq^k\times\{\td{0}_N\})\subseteq\fq^k\times\fq^N$, where $N=\binom{n}{r}$, $I^{\tx{Gb}}_{\tx{supp}}=\ideala{h_{i,S}:i\in[m],S\in\mc{P}_{r+1}(n)}\subseteq$ $ R^{\tx{Gb}}_{\tx{supp}}=R[C_T:T\in\mc{P}_r(n)]$ and $h_{i,S}$ is as in \eqref{eq:SM poly}.
		\item Set $T_0=\{n-r+1,\dots,n\}\in\mc{P}_r(n)$. $V_{\tx{supp}}=\mb{V}(I_{\tx{supp}}^{\tx{Gb}}+\ideala{C_{T_0}-1})$.
	\end{enumerate}
	\subsection{Relations between $I_{\mr{KS}}$ and $I_{\mr{minors}}$}
	\begin{proposition}
		Arrange $f_{i,j}$ in a matrix:
		\begin{equation*}
			F:=\begin{pmatrix}
				f_{1,1}&\cdots&f_{1,n-r}\\
				\vdots&\ddots&\vdots\\
				f_{m,1}&\cdots&f_{m,n-r}
			\end{pmatrix}
			=\td{M_x}\cdot\begin{pmatrix}
				I_{n-r}\\
				\tx{Var}(\td{y})
			\end{pmatrix},\quad\tx{Var}(\td{y})=\begin{pmatrix}
				y_{1,1}&\cdots&y_{1,n-r}\\
				\vdots&\ddots&\vdots\\
				y_{r,1}&\cdots&y_{r,n-r}
			\end{pmatrix}
		\end{equation*}
		We have the following conclusions:
		\begin{enumerate}
			\item $I_{\tx{minors}}\subseteq I_{\tx{KS}}\cap R$.
			\item Assume that $\Delta\neq 0$ is an $r$-minor of $\td{M_x}$ with columns indexed by $\{n-r+1,\dots,n\}$. Then we have $(I_{\tx{minors}})_\Delta=(I_{\tx{KS}})_\Delta\cap R_\Delta$.
		\end{enumerate}
	\end{proposition}

	\begin{proof}
		Suppose that $\td{M_x}=(\td{M_{x,1}},\td{M_{x,2}})$ with $\td{M_{x,1}}\in\tx{Mat}_{m\times (n-r)}(R)$ and $\td{M_{x,2}}\in\tx{Mat}_{m\times r}(R)$. Then 
		\begin{align*}
			\td{M_x}&=\begin{pmatrix}
				\td{M_{x,1}}+\td{M_{x,2}}\cdot\tx{Var}(\td{y})&\td{M_{x,2}}
			\end{pmatrix}\cdot\begin{pmatrix}
				I_{n-r}&0_{(n-r)\times r}\\
				-\tx{Var}(\td{y})&I_r
			\end{pmatrix}\\&=\begin{pmatrix}
				F&\td{M_{x,2}}
			\end{pmatrix}\cdot\begin{pmatrix}
				I_{n-r}&0_{(n-r)\times r}\\
				-\tx{Var}(\td{y})&I_r
			\end{pmatrix}
		\end{align*}
		Hence we have the inclusion $I_{r+1}(\td{M_x})\subseteq I_{r+1}((F,\td{M_{x,2}}))$ (see \cite[pp. 740, (1)]{Lang}). Any $(r+1)\times(r+1)$ submatrix of $(F,\td{M_{x,2}})$ must involve one of the first $n-r$ columns. We get
		\begin{equation*}
			I_{\tx{minors}}=I_{r+1}(\td{M_x})\subseteq I_{r+1}((F,\td{M_{x,2}}))\cap R\subseteq I_1(F)\cap R=I_{\tx{KS}}\cap R.
		\end{equation*}
		
		Now we prove $(2)$. Suppose that $\Delta=|(\td{M_x})_{T_1,T_0}|$ with $T_1\in\mc{P}_r(m)$ and $T_0=\{n-r+1,\dots,n\}$. Let $T_2=\{1,\dots,n-r\}$. Then
		\begin{equation*}
			(\td{M_x})_{T_1,T_0}+(\td{M_x})_{T_1,T_2}
			\tx{Var}(\td{y})=F_{T_1,*}.
		\end{equation*}
		So $y_{i,j}+\frac{v_{i,j}(\td{x})}{\Delta}\in (I_{\tx{KS}})_{\Delta}$, where $v_{i,j}(\td{x})$ is the $(i,j)$-th entry of $\tx{adj}((\td{M_x})_{T_1,T_0})(\td{M_x})_{T_1,T_2}$. By Claim \ref{clm:I cap R},
		\begin{equation*}
			(I_{\tx{KS}})_{\Delta}\cap R_\Delta=\ideala{u_{i,j}(\td{x})-\sum^r_{l=1}u_{i,n-r+l}(\td{x})\frac{v_{l,j}(\td{x})}{\Delta}:i\in[m],j\in[n-r]}.	
		\end{equation*}
		By Claim \ref{clm:minor}, $(-1)^{i-1}v_{i,j}(\td{x})$ is the determinant of the matrix made by removing the $(i+1)$-th column of the $r\times(r+1)$ matrix $(
		((\td{M_x})_{T_1,T_2})_{*,j},(\td{M_x})_{T_1,T_0})$. Note that $((\td{M_x})_{T_1,T_2})_{*,j}$, the $j$-th column of $(\td{M_x})_{T_1,T_2}$, is exactly $(\td{M_x})_{T_1,j}$, as $j\in T_2$. Thus
		\begin{equation*}
			u_{i,j}(\td{x})\Delta-\sum^r_{l=1}u_{i,n-r+l}(\td{x})v_{l,j}(\td{x})=
			\begin{vmatrix}
				u_{i,j}(\td{x})&(\td{M_x})_{i,T_0}\\
				(\td{M_x})_{T_1,j}&(\td{M_x})_{T_1,T_0}
			\end{vmatrix}=\begin{cases}
				0&\tx{if}\;i\in T_1\\
				|(\td{M_x})_{T_1\cup\{i\},T_0\cup\{j\}}|&\tx{otherwise}
			\end{cases}.
		\end{equation*}
		We have $(I_{\tx{KS}})_{\Delta}\cap R_{\Delta}\subseteq (I_{\tx{minors}})_\Delta$. The converse is obvious, so $(I_{\tx{KS}})_{\Delta}\cap R_{\Delta}=(I_{\tx{minors}})_\Delta$.
	\end{proof}
	
	In general, the hypothesis that column indices of $\Delta$ is $\{n-r+1,\dots,n\}$ cannot be removed. We see the following example. Set $m=2,n=4,r=1$ and
	\begin{gather*}
		M_1=\begin{pmatrix}
			1&0&1&0\\0&0&0&0\\
		\end{pmatrix},\quad
		M_2=\begin{pmatrix}
			0&1&0&0\\
			0&0&0&0
		\end{pmatrix},\quad
		M_3=\begin{pmatrix}
			0&0&0&0\\
			1&0&0&1	\end{pmatrix},\\
		\td{M_x}=x_1M_1+x_2M_2+x_3M_3=\begin{pmatrix}
			x_1&x_2&x_1&0\\
			x_3&0&0&x_3
		\end{pmatrix}.
	\end{gather*}
	Then $I_{\tx{minors}}=\ideala{x_1x_3,x_2x_3}\subseteq \fq[x_1,x_2,x_3]$, $I_{\tx{KS}}=\ideala{x_1,x_2,x_3+x_3y_{1,1},x_3y_{1,2},x_3y_{1,3}}\subseteq\fq[x_1,x_2,x_3,$ $y_{1,1},y_{1,2},y_{1,3}]$. $(\td{M_x})_{2,4}=x_3,(\td{M_x})_{1,3}=x_1$. We see that
	\begin{gather*}
		(I_{\tx{KS}})_{x_3}\cap \fq[x_1,x_2,x_3]_{x_3}=(I_{\tx{minors}})_{x_3}=\ideala{x_1,x_2}_{x_3},\\
		\fq[x_1,x_2,x_3]_{x_1}=(I_{\tx{KS}})_{x_1}\cap \fq[x_1,x_2,x_3]_{x_1}\supsetneqq (I_{\tx{minors}})_{x_1}.
	\end{gather*}
	\subsection{Relations between $I_{\mr{supp}}$ and $I_{\mr{KS}}$}\,
	
	In the Support Minors Modeling, we consider the $r$-minors of the $r\times n$ matrix $B_{\td{y}}=(-\tx{Var}(\td{y}),I_r)\in\tx{Mat}_{r\times n}(R_{\tx{KS}})$ as new variables. Its $r$-minors satisfy certain relations. Especially:
	\begin{equation*}
		|(B_{\td{y}})_{*,T}|=\begin{cases}
			1&T=T_0=\{n-r+1,\dots,n\}\\
			(-1)^iy_{i,j}&T=T_{i,j}=T_0\cup\{j\}\setminus\{n-r+i\},1\leq i\leq r,1\leq j\leq n-r
		\end{cases}.
	\end{equation*}
	$\{C_T\}_{T\in\mc{P}_r(n)\setminus\{T_0\}}$ is a set of indeterminates. We retain the notation introduced at the beginning of this section. Set
	\begin{equation*}
		Y_T=\begin{cases}
			1&T=T_0\\
			C_T&T\neq T_0
		\end{cases}.
	\end{equation*}
	Here $Y_T$ corresponds to the $r$-minor $|(B_{\td{y}})_{*,T}|$.
	\begin{align*}
		I_{\tx{supp}}&=\ideala{g_{i,S}(\td{x},\td{Y})=\ssm{j\in S}(-1)^{\tx{sgn}(j,S)}u_{i,j}(\td{x})Y_{S\setminus\{j\}}:i\in[m],S\in\mc{P}_{r+1}(n)}\\
		&\subseteq R_{\tx{supp}}=R[C_T:T\in\mc{P}_r(n)\setminus\{T_0\}].
	\end{align*}
	If $(\td{a},(c_T)_{T\in\mc{P}_r(n)\setminus\{T_0\}})\in V_{\tx{supp}}=\mb{V}(I_{\tx{supp}})$, then $(\td{a},(c_T)_{T\in\mc{P}_r(n)})\in V^{\tx{Gb}}_{\tx{supp}}$, here $c_{T_0}=1$.
	For $j\in$ $[n-r]$, let $S_j=\{j,n-r+1,\dots,n\}\in\mc{P}_{r+1}(n)$. In some sense, if we think $Y_{T_{i,j}}$ as $(-1)^iy_{i,j}$, then the Support Minors polynomials induce the KS polynomials:
	\begin{equation*}
		g_{i,S_j}=u_{i,j}(\td{x})+\sum^r_{l=1}(-1)^lu_{i,n-r+l}(\td{x})Y_{T_{l,j}}\leadsto f_{i,j}=u_{i,j}(\td{x})+\sum^r_{l=1}u_{i,n-r+l}(\td{x})y_{l,j}.
	\end{equation*}
	The difficult part of discussing relations between $I_{\tx{supp}}$ and $I_{\tx{KS}}$ is that	$R_{\tx{supp}}$ and $R_{\tx{KS}}$ have different indeterminates. Bardet et. al \cite{GSM} considered the algebra homomorphism $\Psi:R_{\tx{supp}}\to R_{\tx{KS}}$ by sending $Y_T$ to $|B_{*,T}|$ and proved that $\Psi(I_{\tx{supp}})=I_{\tx{KS}}$. But they did not give the relations between $I_{\tx{supp}}$ and $\Psi(I_{\tx{supp}})$. Actually, $\{\Psi(Y_T)\}_{T\in\mc{P}_r(n)}$ satisfy the Grassmann-Plücker relations, i.e., for any $L\in\mc{P}_{r-1}(n)$ and $S\in\mc{P}_{r+1}(n)$, $\ssm{i\in S\setminus L}(-1)^{\tx{sgn}(i,S)+\tx{sgn}(i,L)}\Psi(Y_{T\cup\{i\}})\Psi(Y_{S\setminus\{i\}})=0$. It is the reason why we apply general considerations on the Grassmann-Plücker polynomials.
	
	For $L\in\mc{P}_{r-1}(n)$ and $S\in\mc{P}_{r+1}(n)$ ($\mc{P}_0(n)=\{\emptyset\}$), the Grassmann-Plücker polynomial with respect to $(L,S)$ is the quadric polynomial in $\fq[C_T:T\in\mc{P}_r(n)\setminus\{T_0\}]$:
	\begin{equation}\label{eq:GS}
		F_{L,S}:=\ssm{i\in S\setminus L}(-1)^{\tx{sgn}(i,S)+\tx{sgn}(i,L)}Y_{L\cup\{i\}}Y_{S\setminus\{i\}},
	\end{equation}
	Let $(c_T)_{T\in\mc{P}_r(n)\setminus T_0}\in\fq^{N-1}$ be a zero of all $F_{L,S}$, we will show that $c_T=(-\tx{Var}(\td{b}),I_r)_T$, where $\tx{Var}(\td{b})=(b_{i,j})_{i\in[r],j\in[n-r]}$ with $b_{i,j}=(-1)^ic_{T_{i,j}}$, from which the result will follow. Consider the matrix
	\begin{equation*}
		B_{\td{Y}}=\begin{pmatrix}
			(-1)^{0}Y_{T_{1,1}}&(-1)^{0}Y_{T_{1,2}}&\cdots&(-1)^{0}Y_{T_{1,n-r}}&1&0&\cdots&0\\
			(-1)^{1}Y_{T_{2,1}}&(-1)^{1}Y_{T_{2,2}}&\cdots&(-1)^{1}Y_{T_{2,n-r}}&0&1&\cdots&0\\
			\vdots&\vdots&\ddots&\vdots&\vdots&\vdots&\ddots&\vdots\\
			(-1)^{r-1}Y_{T_{r,1}}&(-1)^{r-1}Y_{T_{r,2}}&\cdots&(-1)^{r-1}Y_{T_{r,n-r}}&0&0&\cdots&1
		\end{pmatrix},
	\end{equation*}
	We see that $|(B_{\td{Y}})_{*,T_{i,j}}|=Y_{T_{i,j}}$. We claim that:
	
	\begin{lemma}
		The following two ideals in $\fq[C_T:T\in\mc{P}_r(n)\setminus\{T_0\}]$ are equal:
		\begin{equation*}
			I_{\tx{GP}}:=\ideala{F_{L,S}:L\in\mc{P}_{r-1}(n),S\in\mc{P}_{r+1}(n)}=\ideala{Y_T-|(B_{\td{Y}})_{*,T}|:T\in\mc{P}_r(n)}.
		\end{equation*}
		
	\end{lemma}
	\begin{proof}
		$Y_{T_0}-|(B_{\td{Y}})_{*,T_0}|=0$. We show that $Y_T-|(B_{\td{Y}})_{*,T}|\in I_{\tx{GP}}$ for $T\in\mc{P}_r(n)$ by induction on $\#(T\setminus\{T_0\})$. The case $\#(T\setminus T_0)=1$ is easy, as $T=T_{i,j}$ for some $i\in[r]$ and $j\in[n-r]$, we have $Y_{T_{i,j}}-|(B_{\td{Y}})_{*,T_{i,j}}|=0$. 
		Assume that $\#(T\setminus T_0)=l>1$. For $j\in T\setminus T_0$, $|(B_{\td{Y}})_{\{1,\dots,r\}\setminus\{i\},T\setminus\{j\}}|$, the $(r-1)$-minor of the element $(-1)^{i-1}Y_{T_{i,j}}$ in $(B_{\td{Y}})_{*,T}$, is exactly $(-1)^{i+\tx{sgn}(n-r+i,T)}|(B_{\td{Y}})_{*,T\cup\{n-r+i\}\setminus\{j\}}|$. The submatrix $(B_{\td{Y}})_{*,T\cup\{n-r+i\}}$ and the minor can be pictured as follows:
		\begin{figure}[H]
			\centering
			\includegraphics[width=0.4\linewidth]{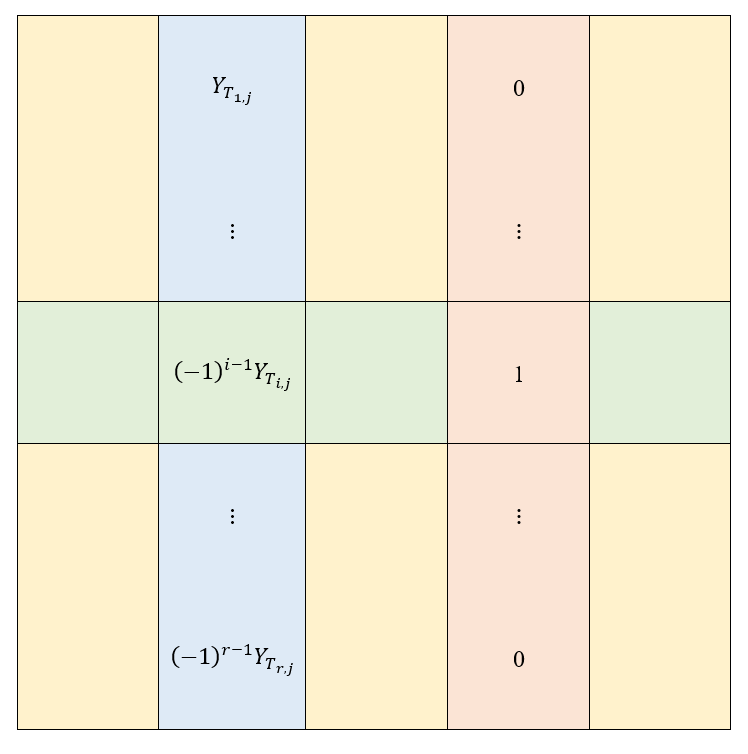}
			\caption{}
			\label{fig:pic}
		\end{figure}
		\noindent The yellow blocks consist of the $(r-1)$-minor of $(-1)^{i-1}Y_{T_{i,j}}$ in $(B_{\td{Y}})_{*,T}$. Expand $|(B_{\td{Y}})_{*,T}|$ according to the $i$-th row, we have
		\begin{equation*}
			|(B_{\td{Y}})_{*,T}|=\ssm{j\in T\setminus T_0}(-1)^{\tx{sgn}(j,T)+\tx{sgn}(n-r+i,T)+i}Y_{T_{i,j}}|(B_{\td{Y}})_{*,T\cup\{n-r+i\}\setminus\{j\}}|.
		\end{equation*}
		Since $j\in T$ and $n-r+i\notin T$, there is $\#((T\cup\{n-r+i\}\setminus\{j\})\setminus T_0)=l-1$. By the induction hypo- thesis, we take $S=T\cup\{n-r+i\}$ and $L=T_0\setminus\{n-r+i\}$, as $\tx{sgn}(n-r+i,L)=i-1$, $\tx{sgn}(j,L)=0$, $\tx{sgn}(n-r+i,S)=\tx{sgn}(n-r+i,T)$ and $\tx{sgn}(j,S)=\tx{sgn}(j,T)$, we obtain
		\begin{align*}
			&(-1)^{\tx{sgn}(n-r+i,L)+\tx{sgn}(n-r+i,S)}(Y_T-|(B_{\td{Y}})_{*,T}|)\\
			=&(-1)^{\tx{sgn}(n-r+i,L)+\tx{sgn}(n-r+i,S)}Y_T+\ssm{j\in T\setminus T_0}(-1)^{\tx{sgn}(j,S)}Y_{L\cup\{j\}}|(B_{\td{Y}})_{*,S\setminus\{j\}}|\\
			=&F_{L,S}-\ssm{j\in T\setminus T_0}(-1)^{\tx{sgn}(j,S)}Y_{L\cup\{j\}}\kh{Y_{S\setminus\{j\}}-|(B_{\td{Y}})_{*,S\setminus\{j\}}|}\in I_{\tx{GP}}.
		\end{align*}
		Hence $Y_T-|(B_{\td{Y}})_{*,T}|\in I_{\tx{GP}}$. So we are done by induction.
		
		Conversely, we can verify that for each $S\in\mc{P}_{r+1}(n)$ and $L\in\mc{P}_{r-1}(n)$,
		\begin{equation*}
			\ssm{i\in S\setminus L}(-1)^{\tx{sgn}(i,S)+\tx{sgn}(i,L)}|(B_{\td{Y}})_{*,L\cup\{i\}}|\cdot|(B_{\td{Y}})_{*,S\setminus\{i\}}|=0,
		\end{equation*}
		see \cite[Lemma 4.4]{Det}. Hence
		\begin{align*}
			F_{L,S}&=\ssm{i\in S\setminus L}(-1)^{\tx{sgn}(i,S)+\tx{sgn}(i,L)}\kh{Y_{L\cup\{i\}}-|(B_{\td{Y}})_{*,L\cup\{i\}}|+|(B_{\td{Y}})_{*,L\cup\{i\}}|}\cdot\kh{Y_{S\setminus\{i\}}-|(B_{\td{Y}})_{*,S\setminus\{i\}}|+|(B_{\td{Y}})_{*,S\setminus\{i\}}|}\\
			&\in\ideala{Y_T-|(B_{\td{Y}})_{*,T}|:T\in\mc{P}_r(n)}.
		\end{align*}
		This implies the assertion.
	\end{proof}
	
	\begin{proposition}
		$(\td{M_x})_{i,*}$ is the $i$-th row of $\td{M_x}$. We form the $(r+1)\times n$ matrix $B_{\td{Y},i}$ by appending to $B_{\td{Y}}$ the top row $(\td{M_x})_{i,*}$, i.e.,
		\begin{equation*}
			B_{\td{Y},i}=\begin{pmatrix}
				(\td{M_x})_{i,*}\\
				B_{\td{Y}}
			\end{pmatrix}.
		\end{equation*}
		Recall that $S_j=\{j,n-r+1,\dots,n\}$ ($j\leq n-r$). We see that
		\begin{equation*}
			|(B_{\td{Y},i})_{*,S_j}|=u_{i,j}(\td{x})+\sum^r_{l=1}(-1)^lu_{i,n-r+l}(\td{x})Y_{T_{l,j}}=g_{i,S_j}\in I_{\tx{supp}}.
		\end{equation*}
		Let $\tld{R}_{\tx{supp}}=R[C_{T_{i,j}}:i\in[r],j\in[n-r]]$. Then
		\begin{equation*}
			I_{\tx{supp}}\cap\tld{R}_{\tx{supp}}=\sum^m_{i=1}I_{r+1}(B_{\td{Y},i})=\ideala{g_{i,S_j}:i\in[m],j\in[n-r]}.
		\end{equation*}
	\end{proposition}
	
	\begin{proof}
		The inclusion $\ideala{g_{1,S_1},\dots,g_{m,S_{n-r}}}\subseteq I_{\tx{supp}}\cap \tld{R}_{\tx{supp}}$ is obvious. For $j\in[n-r]$, there is 
		\begin{equation*}
			(B_{\td{Y},i})_{*,j}=g_{i,S_j}\td{e}_0-\sum^r_{l=1}(-1)^lY_{T_{l,j}}(B_{\td{Y},i})_{*,n-r+l},\quad\td{e}_0=\begin{pmatrix}
				1\\0\\\vdots\\0
			\end{pmatrix}\in\fq^{r+1}.
		\end{equation*}
		By Claim \ref{clm:det}, $I_{r+1}(B_{\td{Y},i})=\ideala{g_{i,S_j}:j\in[n-r]}$. By a slight abuse of notation, we still denote by $I_{\tx{GP}}$ the ideal of $R_{\tx{supp}}$ generated by all the Grassmann–Plücker polynomials $F_{L,S}$ \eqref{eq:GS}. Using Claim \ref{clm:I cap R}, we obtain
		\begin{align*}
			(I_{\tx{supp}}+I_{\tx{GP}})\cap \tld{R}_{\tx{supp}}&=(I_{\tx{supp}}+\ideala{Y_T-|(B_{\td{Y}})_{*,T}|:T\in\mc{P}_r(n)})\cap\tld{R}_{\tx{supp}}\\
			&=\ideala{\ssm{j\in S}(-1)^{\tx{sgn}(j,S)}u_{i,j}(\td{x})|(B_{\td{Y}})_{*,S\setminus\{j\}}|:i\in[m],S\in\mc{P}_{r+1}(n)}\\
			&=\ideala{|(B_{\td{Y},i})_{*,S}|:i\in[m],S\in\mc{P}_{r+1}(n)}=\sum^m_{i=1}I_{r+1}(B_{\td{Y},i}).
		\end{align*}
		This gives that $(I_{\tx{supp}}+I_{\tx{GP}})\cap\tld{R}_{\tx{supp}}=I_{\tx{supp}}\cap\tld{R}_{\tx{supp}}=\sum^m_{i=1}I_{r+1}(B_{\td{Y},i})$.
	\end{proof}
	
	\begin{corollary}
		We construct an $\fq$-algebra isomorphism
		\begin{align*}
			\Phi:\tld{R}_{\tx{supp}}\to R_{\tx{KS}}=\fq[x_1,\dots,x_k,y_{1,1},\dots,y_{r,n-r}].
		\end{align*}
		by sending $x_i$ to $x_i$ and $C_{T_{j,l}}$ to $(-1)^jy_{j,l}$. Then
		\begin{equation*}
			\Phi(I_{\tx{supp}}\cap\tld{R}_{\tx{supp}})=\ideala{\Phi(g_{i,S_j}):i\in[m],j\in[n-r]}=I_{\tx{KS}}.
		\end{equation*}
	\end{corollary}
	
	\subsection{Relations between $I^{\mr{Gb}}_{\mr{supp}}$ and $I_{\mr{minors}}$}\,
	
	We write $\pi$ for the projection of $\fq^k\times\fq^N\to\fq^k$.
	\begin{proposition}\label{pro:Vsup and Vmin}
		$\pi(V^{\tx{Gb}}_{\tx{supp}})=V_{\tx{minors}}$.
	\end{proposition}
	
	\begin{proof}
		if $\td{a}\in V_{\tx{minors}}$, then $\tx{rank}\,\td{M_a}=r_0\leq r$. Choose linearly independent row vectors $\{\td{w}_i\}_{i=1}^{r_0}\in\fq^n$ such that $\{\td{w}_i\}_{i=1}^{r_0}$ span the row space of $\td{M_a}$. Extend $\{\td{w}_i\}_{i=1}^{r_0}$ to a basis $\{\td{w}_i\}_{i=1}^n$ of $\fq^n$. We write $\{\td{w}_i\}_{i=1}^r$ as the rows of the $r\times n$ matrix:
		\begin{equation*}
			W=\begin{pmatrix}
				\td{w}_1\\
				\vdots\\
				\td{w}_r
			\end{pmatrix}.
		\end{equation*}
		For each $i$, we append the $i$-th row of $\td{M_a}$ to get the $(r+1)\times n$ matrix of rank $r$: 
		\begin{equation*}
			W_i=\begin{pmatrix}
				(\td{M_a})_{i,*}\\
				W
			\end{pmatrix}.
		\end{equation*}
		Set $c_T=|W_{*,T}|$ for $T\in\mc{P}_r(n)$. $\tx{rank}\,W=r$, hence $\td{c}=(c_T)_{T\in\mc{P}_r(n)}\neq\td{0}_N$. $\ssm{j\in S}(-1)^{\tx{sgn}(j,S)}u_{i,j}(\td{a})c_{S\setminus\{j\}}$ is the $(r+1)$-minor $|(W_i)_{*,S}|$. But every $W_i$ has rank $r$, so $(\td{a},\td{c})\in V^{\tx{Gb}}_{\tx{supp}}$. $V_{\tx{minors}}\subseteq\pi(V^{\tx{Gb}}_{\tx{supp}})$.
		
		Conversely, fix $(\td{a},\td{c})\in V^{\tx{Gb}}_{\tx{supp}}$. If $c_{T_1}\neq 0$ for some $T_1\in\mc{P}_r(n)$, then columns of $\td{M_a}$ indexed by $T_1$ span the column space of $\td{M_a}$. Indeed, for any $j\notin T$, $h_{i,T_1\cup\{j\}}(\td{a},\td{c})=0$ ($i\in[m]$) implies that
		\begin{equation*}
			(\td{M_a})_{*,j}=(-1)^{\tx{sgn}(j,T_1)}c_{T_1}^{-1}\ssm{l\in T_1}(-1)^{\tx{sgn}(l,T_1\cup\{j\})}c_{T_1\cup\{j\}\setminus\{l\}}(\td{M_a})_{*,l},
		\end{equation*}
		here $(\td{M_a})_{*,j}$ is the $j$-th column of $\td{M_a}$. Thus $\tx{rank}\,\td{M_a}\leq r$. $\pi(V^{\tx{Gb}}_{\tx{supp}})=V_{\tx{minors}}$. 
	\end{proof}

	If $\td{c}=(c_T)_{T\in\mc{P}_r(n)}$ comes from an $r\times n$ matrix $W_0$, i.e., $c_T=$ $|(W_0)_{*,T}|$ for all $T\in\mc{P}_r(n)$. $h_{i,S}(\td{a},\td{c})=0$ for all $S\in\mc{P}_{r+1}(n)$ means that
	\begin{equation*}
		\tx{rank}\,\begin{pmatrix}
			(\td{M_a})_{i,*}\\
			W_0
		\end{pmatrix}=r
	\end{equation*}
	$\td{c}$ is non-zero, hence $\tx{rank}\,W_0=r$. The $i$-th row of $\td{M_a}$ can be linearly expressed as a linear combination of rows of $W_0$. One can verify that if $c_{T_1}=0$ for $T_1\in\mc{P}_r(n)$, then $\tx{rank}\,(\td{M_a})_{*,T_1}<r$. This property remains true for the general case.
	
	\begin{proposition}
		For any $(\td{a},\td{c})\in V^{\tx{Gb}}_{\tx{supp}}$, if $c_{T_1}=0$ for some $T_1\in\mc{P}_r(n)$, then $\tx{rank}\,(\td{M_a})_{*,T_1}<r$.
	\end{proposition}
	
	
	\begin{proof}
		Suppose that $\tx{rank}\,(\td{M_a})_{*,T_1}=r$, then the columns $\{(\td{M_a})_{*,j}\}_{j\in T_1}$ are linearly independent. We state that $c_T=0$ for all $T$. This contradicts $\td{c}\neq\td{0}_N$. So $\tx{rank}\,(\td{M_a})_{*,T_1}<r$. We prove the statement by induction on $\#(T\setminus T_1)$, starting with $\#(T\setminus T_1)=0$, where the statement follows directly from the condition. Suppose that $c_T=0$ for any $T\in\mc{P}_r(n)$ such that $\#(T\setminus T_1)=k_0$. For $T\in\mc{P}_r(n)$ with $\#(T\setminus T_1)=k_0+1$, choose $j\in T\setminus T_1$, set $S=T_1\cup\{j\}$, we have
		\begin{gather*}
			\ssm{l\in S\cap T_1}(-1)^{\tx{sgn}(l,S)}u_{i,l}(\td{a})c_{S\setminus\{l\}}=-\ssm{l\in S\setminus T_1}(-1)^{\tx{sgn}(l,S)}u_{i,l}(\td{a})c_{S\setminus\{l\}},\quad  1\leq i\leq m.\\
			\ssm{l\in S\cap T_1}(-1)^{\tx{sgn}(l,S)}c_{S\setminus\{l\}}(\td{M_a})_{*,l}=-\ssm{l\in S\setminus T_1}(-1)^{\tx{sgn}(l,S)}c_{S\setminus\{l\}}(\td{M_a})_{*,l}.
		\end{gather*}
		When $l\in S\setminus T_1,\#((S\setminus\{l\})\setminus T_1)=k_0$, by induction these $c_{S\setminus\{l\}}=0$. The column vectors $\{(\td{M_a})_{*,l}\}_{l\in T_1}$ are linearly independent, so $c_{S\setminus\{l\}}=0$ for $l\in S\cap T_1$. $c_{S\setminus\{j\}}=c_T=0$. This completes the proof.
	\end{proof}

	\begin{proposition}
		We have the following conclusions:
		\begin{enumerate}
			\item For any subset $T\in\mc{P}_r(n)$, $C_T^{r+1}I_{\tx{minors}}\subseteq I^{\tx{Gb}}_{\tx{supp}}$.
			\item Let $\mf{m}=\ideala{C_T:T\in\mc{P}_r(n)}\subseteq R^{\tx{Gb}}_{\tx{supp}}$, then
			\begin{equation*}
				\mb{I}(V_{\tx{minors}})=I_{\tx{minors}}+\ideala{x_i^q-x_i:i\in [k]}=\kh{\mb{I}(V^{\tx{Gb}}_{\tx{supp}}):\mf{m}}\cap R,
			\end{equation*}
			here $\mb{I}(V^{\tx{Gb}}_{\tx{supp}})=I^{\tx{Gb}}_{\tx{supp}}+\ideala{x_i^q-x_i,C_T^q-C_T:i\in[k],T\in\mc{P}_r(n)}$.
		\end{enumerate}
	\end{proposition}
	
	\begin{proof}
		Choose $j\notin T$, then $h_{i,T\cup\{j\}}=0$ while $i$ ranges over elements of $[m]$. We obtain
		\begin{equation*}
			C_T(\td{M_x})_{*,j}=(-1)^{\tx{sgn}(j,T)}\ssm{l\in T}(-1)^{\tx{sgn}(l,T\cup\{j\})}C_{T\cup\{j\}\setminus\{l\}}(\td{M_x})_{*,l}+
			(-1)^{\tx{sgn}(j,T)}\begin{pmatrix}
				h_{1,T\cup\{j\}}\\
				\vdots\\
				h_{m,T\cup\{j\}}
			\end{pmatrix}.
		\end{equation*}
		For any $(r+1)$-minor $|(\td{M_x})_{S_1,S_2}|$ of $\td{M_x}$ with $S_1\in\mc{P}_{r+1}(m)$ and $S_2\in\mc{P}_{r+1}(n)$,  multiplying every column of the submatrix $(\td{M_x})_{S_1,S_2}$ indexed by $j\in S_2\setminus T$ by $C_T$ and using Claim \ref{clm:minor}, we obtain
		\begin{equation*}
			B_T^{\#(S_2\setminus T)}|(\td{M_x})_{S_1,S_2}|\in\ideala{h_{i,T\cup\{j\}}:i\in S_1,j\in S_2\setminus T}\subseteq I^{\tx{Gb}}_{\tx{supp}}.
		\end{equation*}
		Hence $B_T^{r+1}I_{\tx{minors}}\subseteq I^{\tx{Gb}}_{\tx{supp}}$.
		
		To prove (2): If $f\in\fq[x_1,\dots,x_n]$ satisfies $f\mf{m}\subseteq \mb{I}(V_{\tx{supp}})$, we get $f\in \mb{I}(V_{\tx{minors}})$. This is because for any $\td{a}\in V_{\tx{minors}}$, choose non-zero $\td{c}=(c_T)_{T\in\mc{P}_r(n)}$ such that $(\td{a},\td{c})\in V_{\tx{supp}}$. Assume that $c_{T_1}\neq 0$, then $fC_{T_1}$ vanishes on $(\td{a},\td{c})$ implies that $f(\td{a})=0$. Conversely, from part (1), fix any $B_T$,
		\begin{align*}
			B_T\,\mb{I}(V_{\tx{minors}})&= \sqrt{\ideala{B_T}^{r+1}}\sqrt{\mb{I}(V_{\tx{minors}})}\subseteq\sqrt{\ideala{B_T^{r+1}}\,\mb{I}(V_{\tx{minors}})}\\&=\sqrt{\ideala{B_T^{r+1}}\,(I_{\tx{minors}}+\ideala{x_i^q-x_i:i\in[k]})}\\
			&\subseteq\sqrt{I^{\tx{Gb}}_{\tx{supp}}+\ideala{x_i^q-x_i,B_T^q-B_T:i\in[k],T\in\mc{P}_r(n)}}\\
			&=\sqrt{\mb{I}(V_{\tx{supp}})}=\mb{I}(V_{\tx{supp}}).
		\end{align*}
		Hence we prove the third assertion.
	\end{proof}
	
	\bibliography{Reg.bib}
	\bibliographystyle{alpha}
\end{document}